\documentclass[12pt]{amsart}%

\usepackage{amsmath,amssymb,amsthm,enumitem,verbatim,stmaryrd,xcolor,microtype,graphicx,mathscinet,mathtools,tikz-cd}

\usepackage[T1]{fontenc}
\usepackage[utf8]{inputenc}
\usepackage[english]{babel}
\usepackage[letterpaper, top=3.2cm,bottom=3.2cm,left=3.2cm,right=3.2cm]{geometry}
\usepackage[bookmarksdepth=2,linktoc=page,colorlinks,linkcolor={red!80!black},citecolor={red!80!black},urlcolor={blue!80!black}]{hyperref}

\usepackage{etoolbox}
\makeatletter
\patchcmd{\@startsection}{\@afterindenttrue}{\@afterindentfalse}{}{}             
\patchcmd{\part}{\bfseries}{\bfseries\LARGE}{}{}
\patchcmd{\section}{\scshape}{\bfseries}{}{}\renewcommand{\@secnumfont}{\bfseries} 
\patchcmd{\@settitle}{\uppercasenonmath\@title}{\large}{}{}
\patchcmd{\@setauthors}{\MakeUppercase}{}{}{}
\addto{\captionsenglish}{} 
\makeatother

\usepackage{fancyhdr}

\newtheorem{theorem}{Theorem}[section]
\newtheorem{lemm}[theorem]{Lemma}
\newtheorem{prop}[theorem]{Proposition}
\newtheorem{coro}[theorem]{Corollary}

\theoremstyle{definition}
\newtheorem{defi}[theorem]{Definition}
\newtheorem{example}[theorem]{Example}

\newtheorem{remark}[theorem]{Remark}

\theoremstyle{remark}

\numberwithin{equation}{section}

\DeclareFontFamily{OT1}{pzc}{}                                
\DeclareFontShape{OT1}{pzc}{m}{it}{<-> s * [1.10] pzcmi7t}{}
\DeclareMathAlphabet{\mathpzc}{OT1}{pzc}{m}{it}
\DeclareSymbolFont{sfoperators}{OT1}{bch}{m}{n} 
\DeclareSymbolFontAlphabet{\mathsf}{sfoperators} 
\DeclareSymbolFont{cmletters}{OML}{cmm}{m}{it}              
\DeclareSymbolFont{cmsymbols}{OMS}{cmsy}{m}{n}
\DeclareSymbolFont{cmlargesymbols}{OMX}{cmex}{m}{n}
\DeclareMathSymbol{\myjmath}{\mathord}{cmletters}{"7C}     \let\jmath\myjmath 
\DeclareMathSymbol{\myamalg}{\mathbin}{cmsymbols}{"71}     
\DeclareMathSymbol{\mycoprod}{\mathop}{cmlargesymbols}{"60}
\DeclareMathSymbol{\myalpha}{\mathord}{cmletters}{"0B}     \let\alpha\myalpha 
\DeclareMathSymbol{\mybeta}{\mathord}{cmletters}{"0C}      \let\beta\mybeta
\DeclareMathSymbol{\mygamma}{\mathord}{cmletters}{"0D}     \let\gamma\mygamma
\DeclareMathSymbol{\mydelta}{\mathord}{cmletters}{"0E}     \let\delta\mydelta
\DeclareMathSymbol{\myepsilon}{\mathord}{cmletters}{"0F}   \let\epsilon\myepsilon
\DeclareMathSymbol{\myzeta}{\mathord}{cmletters}{"10}      \let\zeta\myzeta
\DeclareMathSymbol{\myeta}{\mathord}{cmletters}{"11}       \let\eta\myeta
\DeclareMathSymbol{\mytheta}{\mathord}{cmletters}{"12}     \let\theta\mytheta
\DeclareMathSymbol{\myiota}{\mathord}{cmletters}{"13}      \let\iota\myiota
\DeclareMathSymbol{\mykappa}{\mathord}{cmletters}{"14}     \let\kappa\mykappa
\DeclareMathSymbol{\mylambda}{\mathord}{cmletters}{"15}    \let\lambda\mylambda
\DeclareMathSymbol{\mymu}{\mathord}{cmletters}{"16}        \let\mu\mymu
\DeclareMathSymbol{\mynu}{\mathord}{cmletters}{"17}        \let\nu\mynu
\DeclareMathSymbol{\myxi}{\mathord}{cmletters}{"18}        \let\xi\myxi
\DeclareMathSymbol{\mypi}{\mathord}{cmletters}{"19}        \let\pi\mypi
\DeclareMathSymbol{\myrho}{\mathord}{cmletters}{"1A}       \let\rho\myrho
\DeclareMathSymbol{\mysigma}{\mathord}{cmletters}{"1B}     \let\sigma\mysigma
\DeclareMathSymbol{\mytau}{\mathord}{cmletters}{"1C}       \let\tau\mytau
\DeclareMathSymbol{\myupsilon}{\mathord}{cmletters}{"1D}   \let\upsilon\myupsilon
\DeclareMathSymbol{\myphi}{\mathord}{cmletters}{"1E}       \let\phi\myphi
\DeclareMathSymbol{\mychi}{\mathord}{cmletters}{"1F}       \let\chi\mychi
\DeclareMathSymbol{\mypsi}{\mathord}{cmletters}{"20}       \let\psi\mypsi
\DeclareMathSymbol{\myomega}{\mathord}{cmletters}{"21}     \let\omega\myomega
\DeclareMathSymbol{\myvarepsilon}{\mathord}{cmletters}{"22}\let\varepsilon\myvarepsilon
\DeclareMathSymbol{\myvartheta}{\mathord}{cmletters}{"23}  \let\vartheta\myvartheta
\DeclareMathSymbol{\myvarpi}{\mathord}{cmletters}{"24}     \let\varpi\myvarpi
\DeclareMathSymbol{\myvarrho}{\mathord}{cmletters}{"25}    \let\varrho\myvarrho
\DeclareMathSymbol{\myvarsigma}{\mathord}{cmletters}{"26}  \let\varsigma\myvarsigma
\DeclareMathSymbol{\myvarphi}{\mathord}{cmletters}{"27}    \let\varphi\myvarphi

\DeclareMathOperator{\Supp}{Supp}
\DeclareMathOperator{\Pf}{Pf}

\DeclareMathOperator{\sign}{sgn}

\newcommand\F{{\mathbb F}}
\newcommand\N{{\mathbb N}}

\newcommand\R{{\mathbb R}}

\newcommand\Z{{\mathbb Z}}

\newcommand\cB{{\mathcal B}}

\newcommand\fm{{\mathfrak m}}

\newcommand\bfP{{\bf P}}
  
\renewcommand\geq{\geqslant}
\renewcommand\leq{\leqslant}

\title{Representability of orthogonal matroids over partial fields}

\author{Matthew Baker}
\email{mbaker@math.gatech.edu}
\address{School of Mathematics, Georgia Institute of Technology, Atlanta, USA}

\author{Tong Jin}
\email{tongjin@gatech.edu}
\address{School of Mathematics, Georgia Institute of Technology, Atlanta, USA}

\thanks{The authors' research was supported by a Simons Foundation Travel Grant and NSF Research Grant DMS-2154224.
We thank Tianyi Zhang for helpful discussions on the proof of Theorem~\ref{theorem:G-P}.}

\begin{document}

\maketitle

\begin{abstract}
Let $r \leq n$ be nonnegative integers, and let $N = \binom{n}{r} - 1$. For a matroid $M$ of rank $r$ on the finite set $E = [n]$ and a partial field $k$ in the sense of Semple--Whittle, it is known that the following are equivalent: (a) $M$ is representable over $k$; (b) there is a point $p = (p_J) \in \bfP^N(k)$ with support $M$ (meaning that $\Supp(p) := \{J \in \binom{E}{r} \; \vert \; p_J \ne 0\}$ of $p$ is the set of bases of $M$) satisfying the Grassmann-Pl{\"u}cker equations; and (c) there is a point $p = (p_J) \in \bfP^N(k)$ with support $M$ satisfying just the 3-term Grassmann-Pl{\"u}cker equations. Moreover, by a theorem of P. Nelson, almost all matroids (meaning asymptotically 100\%) are not representable over any partial field. We prove analogues of these facts for Lagrangian orthogonal matroids in the sense of Gelfand--Serganova, which are equivalent to even Delta-matroids in the sense of Bouchet.  
\end{abstract}


\section{Introduction}\label{section:introduction}

For simplicity, throughout this introduction $k$ will denote a field, but all statements remain true when $k$ is a {\em partial field} in the sense of Semple and Whittle~\cite{SW96}, and the proofs will be written in that generality.

\medskip

Let $E$ be a finite set of size $n$, which for concreteness we will sometimes identify with the set $[n] := \{ 1,2,\ldots,n \}$. 
Let $r$ be a nonnegative integer, and let $N = \binom{n}{r} - 1$. 
Let $\binom{E}{r}$ denote the family of all $r$-subsets of $E$. 
We will be considering the projective space $\bfP^N(k)$ with coordinates indexed by the $r$-element subsets of $E$.

\medskip

Let $A$ be an $r \times n$ matrix of rank $r$ whose columns are indexed by $E$. Define $\Delta: \binom{E}{r} \to k$ by $\Delta(J) = \det A_J$, where $A_J$ is the $r \times r$ maximal square submatrix whose set of columns is $J$. We can extend the map $\Delta$ to all subsets of $E$ of size at most $r$ by setting $\Delta(J) = 0$ if $|J| < r$. For every $S \in \binom{E}{r+1}$, $T \in \binom{E}{r-1}$, and $x \in S$, we define $\sign(x; S, T)$ to be $(-1)^m$, where $m$ is the number of elements $s \in S$ with $s > x$ plus the number of elements $t \in T$ with $t > x$. Basic properties of determinants imply that the point $p = (p_J)_{J \in \binom{E}{r}} \in \bfP^N(k)$ defined by $p_J := \Delta(J)$ satisfies the following homogeneous quadratic equations, 
called the {\em Pl{\"u}cker equations} (cf. e.g. \cite[\S{4.3}]{MS15}): 

\begin{equation} \label{eq:GP}
\sum_{x \in S} \sign(x; S, T) X_{S \backslash \{x \}} X_{T \cup \{x\}} = 0. 
\end{equation}

When $S \backslash T = \{i < j < k\}$, we may assume without loss of generality that $T \backslash S = \{\ell\}$ and $\ell > k$. Then we obtain the {\emph{$3$-term Pl{\"u}cker} equations}, which are of particular importance: 
\begin{equation} \label{eq:3termGP}
X_{S \backslash \{i\}}X_{T \cup \{i\}} - X_{S \backslash \{j\}}X_{T \cup \{j\}} + X_{S \backslash \{k\}}X_{T \cup \{k\}} = 0. 
\end{equation}

\medskip

The following result is fundamental:

\begin{theorem}\label{theorem:G-P}
Let $k$ be a field. 
The following are equivalent for a point $p = (p_J) \in \bfP^N(k)$: 
\begin{enumerate}
\item There exists an $r \times n$ matrix $A$ of rank $r$ with entries in $k$ such that $p_J = \det(A_J)$ for any $J \in \binom{E}{r}$. 
\item The point $p$ satisfies the Pl{\"u}cker equations. 
\item The support $\Supp(p) = \{J \in \binom{E}{r} \; \vert \; p_J \ne 0\}$ of $p$ is the set of bases of a matroid of rank $r$ on $E$, and $p$ satisfies the $3$-term Pl{\"u}cker equations. 
\end{enumerate}
\end{theorem}

The equivalence of (1) and (2) in the theorem is just the well-known classical fact that the Grassmannian variety of $r$-dimensional subspaces of a fixed $n$-dimensional vector space is defined (set-theoretically) by the Pl{\"u}cker equations.
The equivalence of (2) and (3) is a folklore fact, but we are not aware of a published reference which furnishes a direct proof, so we give one in Section~\ref{section:G-P} below.

\medskip

One of the interesting features of Theorem~\ref{theorem:G-P} is that it is a `purely algebraic' fact about matrices over a field but its statement involves the combinatorial notion of a matroid. Recall that a {\emph{matroid}} $M$ on $E$ is a nonempty collection $\cB$ of subsets of $E$ satisfying the following  
{\em exchange axiom}: 

\begin{quote}
If $B_1, B_2 \in \cB$, then for any $x \in B_1 \backslash B_2$, there exists an element $y \in B_2 \backslash B_1$ such that $(B_1 \backslash \{x\} ) \cup \{y\}$ belongs to $\cB$. 
\end{quote}

\medskip

This turns out to be equivalent to the {\em a priori} stricter {\em strong exchange axiom}:

\begin{quote}
If $B_1, B_2 \in \cB$, then for any $x \in B_1 \backslash B_2$, there exists an element $y \in B_2 \backslash B_1$ such that $(B_1 \backslash \{x\} ) \cup \{y\}$ and 
$(B_2 \backslash \{y\} ) \cup \{x\}$ both belong to $\cB$. 
\end{quote}

\medskip

The set $E$ is called the {\emph{ground set}} of $M$, and the members of $\cB$ are called the {\emph{bases}}. 
All bases of a matroid $M$ have the same cardinality, called the {\emph{rank}} of $M$.

\medskip

A matroid whose bases are the support of some point $p = (p_J) \in \bfP^N(k)$ satisfying the Pl{\"u}cker equations (or, equivalently, the 3-term Pl{\"u}cker equations) is called {\em representable} over $k$.

\medskip

It is natural to ask whether a `typical' matroid is representable over some field. The answer turns out to be no.
This follows by combining the following two estimates:

\begin{theorem}[Knuth \cite{Kn74}]\label{theorem:numberofmatroids}
Denote by $m_n$ the number of isomorphism classes of matroids on ground set $[n] = \{1, 2, \dots, n\}$. Then 
\[
\log \log m_n \geq n - \frac{3}{2}\log n - O(1). 
\]
\textup{(}Here $\log$ is taken to base two.\textup{)}
\end{theorem}

\begin{theorem}[Nelson~\cite{Ne18}]
For $n \geq 12$, the number $r_n$ of isomorphism classes of matroids on the ground set $[n]$ which are representable over some field satisfies
\[
\log r_n \leq n^3 / 4.
\]
\end{theorem}

Combining these two estimates, we obtain:

\begin{theorem}\label{theorem:not-representableM}
Asymptotically 100\% of all matroids are not representable over any field.
\end{theorem}

One can generalize the classical Grassmannian varieties by the Lagrangian orthogonal Grassmannians $OG(n, 2n)$, which parameterise all $n$-dimensional maximal isotropic subspaces of a given $2n$-dimensional vector space endowed with a symmetric, non-degenerate bilinear form. The combinatorial counterpart of this point of view is the notion of a \emph{Lagrangian orthogonal matroid}, also known as an {\em even Delta-matroid}~\cite{Bo89}. For simplicity, we omit the adjective `Lagrangian' and refer to such objects as \emph{orthogonal matroids}. 

\begin{defi}
Denote by $X\Delta Y$ the symmetric difference of two sets $X,Y$. An {\emph{orthogonal matroid}} on $E$ is a nonempty collection $\cB$ of subsets of $E$ satisfying the following axiom: 
if $B_1, B_2 \in \cB$, then for any $x_1 \in B_1 \Delta B_2$, there exists an element $x_2 \in B_2 \Delta B_1$ with $x_2 \ne x_1$ such that $B_1 \Delta\{x_1, x_2\}  \in \cB$. 
\end{defi}

Like the usual Grassmannian, the Lagrangian orthogonal Grassmannian $OG(n, 2n)$ is also a projective variety cut out by homogeneous quadratic polynomials, known in this case as the {\emph{Wick equations}} \cite{Ri12,We93} (see equations \eqref{eq:Wick} below for a precise formulation). 
The simplest Wick equations have precisely four non-zero terms.

\medskip

By analogy with Theorem~\ref{theorem:G-P}, we will prove:

\begin{theorem}\label{theorem:Wick}
Let $k$ be a field.\footnote{In \S\ref{section:Wick}, we will generalize this fact, and the statement of Theorem~\ref{theorem:Wick}, to partial fields $k$.} 
Let $n \in \N$, set $E = [n]$ and $N = 2^n - 1$, and consider the projective space $\bfP^N(k)$ with coordinates indexed by the subsets of $[n]$.
The following are equivalent for a point $p = (p_J) \in \bfP^N(k)$.

\begin{enumerate}
\item There exists a skew-symmetric matrix $A$ over $k$ with rows and columns indexed by $E$ and a subset $T \subseteq E$ such that $p_J = \Pf(A_{J \Delta T})$ for all $J \subseteq E$. 
\textup{(}Here $A_J$ denotes the $|J| \times |J|$ square submatrix whose sets of row and column indices are both $J$, and $\Pf$ denotes the Pfaffian of a skew-symmetric matrix.\textup{)}
\item The point $p$ satisfies the Wick equations. 
\item The support of $p$ is the set of bases of an orthogonal matroid on $E$, and $p$ satisfies the $4$-term Wick equations. 
\end{enumerate}
\end{theorem}

In particular, a point $p \in \bfP^{N}(k)$ belongs to $OG(n, 2n)$ if and only if there is a subset $T \subseteq E$ such that $\Supp(p) \Delta T$ is the set of bases of an orthogonal matroid and $p$ satisfies the $4$-term Wick relations. 

\medskip

If $M$ is an orthogonal matroid, we say that $M$ is {\em representable} over $k$ if there is a skew-symmetric matrix $A$ over $k$ with rows and columns indexed by $E$ and a subset $T \subseteq E$ such that $p_J = \Pf(A_{J \Delta T})$ for all $J \subseteq E$. 

\medskip

By analogy with Nelson's theorem, we estimate the number of representable orthogonal matroids and show:

\begin{theorem}\label{theorem:not-representableOM}
Asymptotically 100\% of orthogonal matroids are not representable over any field.
\end{theorem}

\section{Representations of Matroids over Partial Fields}\label{section:G-P}

Partial fields are generalizations of fields which have proven to be very useful for studying representability of matroids. 
They were originally introduced by Semple and Whitte~\cite{SW96}, but the definitions below are from~\cite{PvZ13}.\footnote{In \cite{BL21}, one finds a slightly different definition of partial fields from the one in \cite{PvZ13} (there is an additional requirement that $G$ generates $R$ as a ring), but the difference is irrelevant for our purposes in this paper.}

\begin{defi}
A {\emph{partial field}} $P$ is a pair $(G, R)$ consisting of a commutative ring $R$ with $1$ and a subgroup $G$ of the group of units of $R$ such that $-1$ belongs to $G$. 
We say $p$ is an {\emph{element}} of $P$ and write $p \in P$ if $p \in G \cup \{0\}$. 
\end{defi}

\begin{example}
A partial field with $G = R \backslash \{0\}$ is the same thing as a field. 
\end{example}

\begin{example}
The partial field $\F^{\pm}_1 = (\{1, -1\}, \Z)$ is called the {\emph{regular partial field}}. 
\end{example}

\begin{defi}
Let $P = (G, R)$ be a partial field. A {\emph{strong $P$-matroid}} of rank $r$ on $E = [n]$ is a projective solution $p = (p_J) \in {\bfP}^N(P)$ to the Pl{\"u}cker equations \eqref{eq:GP}, i.e., $p_J \in P$ for all $J \in \binom{E}{r}$, not all $P_J$ are zero, and $p$ satisfies \eqref{eq:GP} viewed as equations over $R$.  
A {\emph{weak $P$-matroid}} of rank $r$ on $E = [n]$ is a projective solution $p = (p_J)$ to the $3$-term Pl{\"u}cker equations \eqref{eq:3termGP} such that 
$\Supp(p) := \{J \; \vert \; p_J \neq 0 \}$ is the set of bases of a matroid of rank $r$ on $E$.
\end{defi}

\begin{remark}
Let $P$ be a partial field. If $M$ is a strong or weak $P$-matroid, corresponding to a Pl{\"u}cker vector $p \in {\bfP}^N(P)$, then $\Supp(p) := \{J \; \vert \; p_J \neq 0 \}$ is the set of bases of a matroid 
$\underline{M}$ on $E$, called the {\em underlying matroid} of $M$. 
We say that a matroid $\underline{M}$ is {\emph{$P$-representable}} (or {\emph{representable over $P$}}) if it is the support of a strong (or, equivalently, by the following theorem, weak) $P$-matroid. 
\end{remark}

The following is a generalization of Theorem~\ref{theorem:G-P} to partial fields:

\begin{theorem}\label{theorem:partialfieldG-P}
Let $P = (G, R)$ be a partial field and let $0 \leq r \leq n \in \N$. Let $N = \binom{n}{r} - 1$. Then the following are equivalent for a nonzero point $p = (p_J) \in \bfP^N(P)$: 
\begin{enumerate}
\item There exists a matrix $A$ with entries in $P$ such that $p_J = \det(A_J)$ for all $J \in \binom{[n]}{r}$. 
\item $p$ satisfies the Pl{\"u}cker equations \eqref{eq:GP}, i.e., $p$ is the Pl{\"u}cker vector of a strong $P$-matroid. 
\item $p$ is the Pl{\"u}cker vector of a weak $P$-matroid. 
\end{enumerate}
\end{theorem}

\begin{proof}
It follows from standard properties of determinants for matrices over commutative rings that (1) implies (2), and (2) implies (3) is true by definition. It remains to show that (3) implies (1).
The idea is that given $p: \binom{[n]}{r} \to P$ whose support is the set of bases of a matroid $M$, we will explicitly construct an $r \times n$ matrix $A$ over $P$ such that $\det(A_J) = p(J)$ for all $r$-element subsets $J \subseteq [n]$. 

Without loss of generality, by relabeling the elements of $E$ and rescaling the projective vector $p$ if necessary, we may assume that $[r] = \{1, 2, \dots, r\}$ is a basis of $M$ and that $p([r]) = 1$. 
We define 
\[
A = (I_r \; \vert \; a_{ij})_{1 \leq i \leq n, r+1 \leq j \leq n}, 
\]
where $a_{ij} = (-1)^{r+i}p([r] \backslash \{i\} \cup \{j\})$. We claim that $\det(A_J) = p(J)$ for all $J \in \binom{[n]}{r}$. The proof is by induction on $v_J := r - |J \cap [r]|$. 

If $v_J = 0$, then $J = [r]$ and $\det(A_J) = p_J = 1$. If $v_J = 1$, $J = [r] \backslash \{i\} \cup \{j\}$ for some $i \in [r]$ and $j \in [n] \backslash [r]$, and elementary properties of determinants give
\[
\det(A_J) = (-1)^{i+r}a_{ij} = p(J). 
\]

Now suppose $v_J = l \geq 2$. Then $k := |J \cap [r]| \leq r - 2$. Fix $a < b \in J \backslash [r]$. We wish to show that $\det(A_J) = p(J)$. 

\medskip

{\bf{Case 1. }}Suppose $J$ is a basis of $M$. Then by the basis exchange property, there exists $i \in [r] \backslash J$ such that $B' := J \backslash \{a\} \cup \{i\}$ is a basis. 
By the basis exchange property again, there exists $j \in [r] \backslash B'$ such that $B'' := B' \backslash \{b\} \cup \{j\}$ is also a basis. 
Without loss of generality, we may assume that $i < j$, and then applying the $3$-term Grassmann-Pl{\"u}cker relations to $S = J \backslash \{b\} \cup \{i, j\}$ and $T = J \backslash \{a\}$ (so that $S \backslash T = \{i < j < a\}$), we obtain 
\[
p(S \backslash \{i\}) p(T \cup \{i\}) - p(S \backslash \{j\}) p(T \cup \{j\}) + p(S \backslash \{a\}) p(T \cup \{a\}) = 0. 
\]
Since 
\[
|(S \backslash \{i\}) \cap [r]| = |(T \cup \{i\}) \cap [r]| = |(S \backslash \{j\}) \cap [r]| = |(T \cup \{j\}) \cap [r]| = k+1
\]
and $|(S \backslash \{a\}) \cap [r]| = k+2$, the inductive hypothesis implies that 
\[
\det(A_{S \backslash \{i\}}) \det(A_{T \cup \{i\}}) - \det(A_{S \backslash \{j\}}) \det(A_{T \cup \{j\}}) + \det(A_{S \backslash \{a\}}) p(T \cup \{a\}) = 0. 
\]
Moreover, since $S \backslash \{a\} = B''$ is a basis, $p(S \backslash \{a\}) = \det(A_{S \backslash \{a\}}) \ne 0$. This gives $p(T \cup \{a\}) = \det(A_{T \cup \{a\}})$ as desired. 

\medskip

{\bf{Case 2. }}Suppose $J$ is not a basis of $M$, i.e., $p(J)=0$. Note that if there exist distinct $i, j \in [r] \backslash J$ such that $J \backslash \{a, b\} \cup \{i, j\}$ is a basis, then the proof from Case 1 still works. 
Therefore, we may assume that no such $i$ and $j$ exist. Then for every $i \in [r] \backslash J$, $J_i' :=  J \backslash \{a\} \cup \{i\}$ is not a basis and $\det(A_{J_i'}) = p(J_i') = 0$. 
By $(1) \Rightarrow (2)$ applied with $S = [r] \cup \{a\}$ and $T = J \backslash \{a\}$, we have 
\begin{equation} \label{eq:PluckerA}
\begin{split}
\sign(a; S, T) \det(A_{[r]}) \det(A_J) & + \sum_{i \in [r] \cap J} \sign(i; S, T) \det(A_{S \backslash\{i\}}) \det(A_{T \cup \{i\}}) \\
& + \sum_{i \in [r] \backslash J}\sign(i; S, T)\det(A_{S \backslash\{i\}}) \det(A_{T \cup \{i\}})\\
& = 0. 
\end{split}
\end{equation}

If $i \in [r] \cap J$, then $T \cup \{i\} = T$ so $\det(A_{T \cup \{i\}}) = 0$. If $i \in [r] \backslash J$, then $T \cup \{i\} = J_i'$ so $\det(A_{J_i'}) = 0$. Together with \eqref{eq:PluckerA},
this forces $\det(A_J) = 0 = p_J$. 
\end{proof}

We now explain briefly how to see that Theorem~\ref{theorem:not-representableM} (Nelson's theorem) implies that asymptotically 100\% of all matroids are not representable over any partial field.

\begin{defi}
Let $P_1 = (G_1, R_1)$ and $P_2 = (G_2, R_2)$ be partial fields. A map $\varphi: R_1 \to R_2$ is called a {\emph{homomorphism}} of partial fields if $\varphi$ is a ring homomorphism and $\varphi(G_1) \subseteq G_2$. 
\end{defi}

Matroid representability over partial fields is preserved by homomorphisms. More precisely:

\begin{prop}[Semple-Whittle~\cite{SW96}]
Let $\varphi: P_1 \to P_2$ be a homomorphism of partial fields. If a matroid $M$ is $P_1$-representable, then $M$ is also $P_2$-representable. 
\end{prop}

On the other hand, we also have: 

\begin{lemm} \label{lemma:homomtofield}
If $P=(G,R)$ is a partial field, there exists a homomorphism $P \to k$ for some field $k$.
\end{lemm}

\begin{proof}
Take a maximal ideal $\fm \subseteq R$ and consider the field $k := R/\fm$. Then the natural quotient homomorphism $R \twoheadrightarrow k$ induces a homomorphism of partial fields $P \to k$.
\end{proof}

We deduce:

\begin{coro}
If a matroid is representable over a partial field $P$, then it is representable over a field $k$. 
\end{coro}

Combining this fact with Theorem~\ref{theorem:not-representableM}, we obtain:

\begin{coro}
Asymptotically 100\% of all matroids are not representable over any partial field.
\end{coro}

\section{Representations of Orthogonal Matroids}\label{section:OM}

In this section, we provide some background on orthogonal matroids and their representations.

\medskip

Recall from Section~\ref{section:introduction} that an {\em orthogonal matroid} on $E$ is a nonempty collection $\cB$ of subsets of $E$ satisfying the symmetric exchange axiom. 
As with their matroid counterparts, this axiom can be replaced with an {\em a priori} stricter {\emph{strong symmetric exchange axiom}} (see, for example, Theorem 4.2.4 of~\cite{BGW03}): 

\begin{prop}[Strong Symmetric Exchange]
If $M$ is an orthogonal matroid, then for any $B_1, B_2 \in \cB$ and $x_1 \in B_1 \Delta B_2$, there exists $x_2 \in \cB$ with $x_2 \ne x_1$ such that both $B_1 \Delta \{x_1, x_2\}$ and $B_2 \Delta \{x_1, x_2\}$ belong to $\cB$. 
\end{prop}

\begin{example}
Every matroid is also an orthogonal matroid. In fact, matroids are by definition just orthogonal matroids whose bases all have the same cardinality. 
\end{example}

From the definition, one sees easily that any two bases of an orthogonal matroid have the same parity. 

\begin{example}\label{exm:representable}
Let $E = [4]$ and consider $\cB = \{\emptyset, \{1, 2\}, \{1, 4\}, \{2, 4\}\}$. This is an orthogonal matroid. If we keep the same ground set $E$ and replace every member $B$ of $\cB$ with $B \Delta \{3\}$, we obtain another orthogonal matroid $M' = (E, \cB')$ where $\cB' = \{\{3\}, \{1, 2, 3\}, \{1, 3, 4\}, \{2, 3, 4\}\}$. This is an example of a general operation on orthogonal matroids called {\em twisting}.
\end{example}

The determinant of a matrix $A$ admits a refinement for skew-symmetric matrices called the {\em Pfaffian}. 
The Pfaffian $\Pf(A)$ can be defined recursively as follows. 

\medskip

By convention, we define the Pfaffian of the empty matrix to be $1$. 
Now let $A$ be an $n \times n$ skew-symmetric matrix over a ring $R$, where $n \geq 1$. If $n$ is odd, we set $\Pf(A) := 0$. 
If $n = 2$, we have
\[
\Pf
\begin{pmatrix}
0 & a \\
-a & 0
\end{pmatrix}
 := a. 
\]

Finally, if $n \geq 4$, we set
\[
\Pf(A) := \sum_{j = 2}^n (-1)^j \Pf(A_{\{1, j\}}) \Pf(A_{\{2, \dots, j-1, j+1, \dots, n\}}),
\]
where $A_J$ denotes the $|J| \times |J|$ square submatrix whose rows and columns are indexed by $J$.

\medskip

If $A$ is a $(2k) \times (2k)$ skew-symmetric matrix of indeterminates, $\Pf(A)$ is a homogeneous polynomial of degree $d = 2k-1$ whose coefficients all belong to $\{ 0,1,-1 \}$. A basic fact about the Pfaffians is that $(\Pf(A))^2 = \det(A)$ for every skew-symmetric matrix $A$~\cite{Ca49}. 

\begin{prop}[Wenzel, Prop. 2.3 of \cite{We93}] \label{proposition:Pfaffian}
Let $A$ be an $n \times n$ skew-symmetric matrix over a ring $R$. Let $N = 2^n - 1$. The point $(p_J)_{J \subseteq E} \in \bfP^N(R)$ defined by $p_J = \Pf(A_J)$ satisfies the homogeneous quadratic polynomial equations
\begin{equation} \label{eq:Wick}
\sum_{j = 1}^k (-1)^j \cdot X_{J_1 \Delta\{i_j\}} \cdot X_{J_2 \Delta \{i_j\}} = 0,
\end{equation}
called the {\em Wick equations}\footnote{These identities are known to physicists as Wick's theorem~\cite{Mu99}. We follow \cite{Ri12} rather than \cite{We93} in our terminology.}, for all
$J_1, J_2 \subseteq [n]$ and $J_1 \Delta J_2 = \{i_1 < \cdots < i_k\}$.
\end{prop}

We are especially interested in the shortest possible Wick equations, where $|J_1 \Delta J_2| = 4$, called the {\emph{$4$-term Wick equations}}. 
Concretely, if $J \subseteq [n]$ and $a < b < c < d \in [n] \backslash J$ are distinct, we have: 
\begin{gather*}
X_{Jabcd}X_{J} - X_{Jab}X_{Jcd} + X_{Jac}X_{Jbd} - X_{Jad}X_{Jbc} = 0, \\
X_{Jabc}X_{Jd} - X_{Jabd}X_{Jc} + X_{Jacd}X_{Jb} - X_{Jbcd}X_{Ja} = 0. 
\end{gather*}

Here, and from now on, $Ja$ means $J \cup \{a\}$ in order to simplify the notation. 

\medskip 

If $P$ is a partial field, we may consider the projective solutions in $\bfP^N(P)$ to the different kinds of Wick equations. 

\begin{defi}
A {\emph{strong orthogonal matroid}} over $P$ is a projective solution $p = (p_J)$ to the Wick equations \eqref{eq:Wick}. 
A {\emph{weak orthogonal matroid}} over $P$ is a projective solution $p = (p_J)$ to $4$-term Wick equations whose support $\Supp(p) = \{J \; \vert \; p_J \neq 0 \}$ is the set of bases of an orthogonal matroid on $E$.
\end{defi}

\begin{remark}
One can generalize the definition of weak and strong orthogonal matroids over $P$ from partial fields to tracts in the sense of~\cite{BB19} and obtain cryptomorphic descriptions of these objects along the lines of {\em loc. cit.}, see~\cite{JK23}.
\end{remark}

\begin{prop}[Wenzel, Theorem 2.2 of \cite{We93}] \label{proposition:normalOM}
Let $P = (G, R)$ be a partial field. Given a strong orthogonal matroid over $P$ with $p_\emptyset = 1$, there exists an $n \times n$ skew-symmetric matrix $A$ over $R$ such that $\Pf(A_J) = p_J$. 
\end{prop}

\begin{prop}\label{proposition:supp}
Let $P$ be a partial field. The support of every strong orthogonal matroid over $P$ is the set of bases of an orthogonal matroid. 
\end{prop} 

\begin{proof}
If $p_\emptyset = 1$, then $\Supp(p)$ gives an orthogonal matroid by Theorem 3.3 of \cite{We93}. Otherwise, let $p \in {\bfP^N(P)}$ be the Wick vector of a strong orthogonal matroid over $P$ (i.e., a point of the Lagrangian orthogonal Grassmannian $OG(n, 2n)$ over $P$), and choose $T \ne \emptyset$ such that $p_T = 1$. 
Consider the point $q \in \bfP^N(P)$ whose coordinates are defined by $q_J = p_{J \Delta T}$. We claim that $q$ also satisfies the Wick equations. 
In fact, for any $J_1, J_2 \subseteq [n]$ with $J_1 \Delta J_2 = (J_1 \Delta T) \Delta (J_2 \Delta T) = \{i_1 < \cdots < i_k\}$, we have 
\[
\begin{split}
\sum_{j = 1}^k (-1)^j \cdot q_{J_1 \Delta \{i_j\}} \cdot q_{J_2 \Delta \{i_j\}} & = \sum_{j = 1}^k (-1)^j \cdot p_{J_1 \Delta \{i_j\} \Delta T} \cdot q_{J_2 \Delta \{i_j\} \Delta T}\\
& = \sum_{j = 1}^k (-1)^j \cdot p_{(J_1 \Delta T) \Delta \{i_j\}} \cdot p_{(J_2 \Delta T) \Delta \{i_j\}} \\
& = 0. 
\end{split}
\]
Since $\emptyset \in \Supp(q)$, this gives an orthogonal matroid $M$ with set of bases $\Supp(p) \Delta T$. Therefore, $\Supp(p)$ is the set of bases for the twist $M \Delta T$. 
\end{proof}

\begin{defi}
Let $M$ be an orthogonal matroid, and let $P$ be a partial field. Then $M$ is {\emph{$P$-representable}} if there exists a skew-symmetric matrix $A = (a_{ij})_{i, j \in E}$ with entries in $P$ and a subset $T \subseteq E$ such that 
\[
\cB \Delta T := \{B \Delta T \; \vert \; B \in \cB\} = \{J \subseteq E \; \vert \; \Pf(A_J) \neq 0 \}. 
\]
\end{defi}

\begin{example}
The two orthogonal matroids in Example~\ref{exm:representable} are both $\R$-representable. To see this, let 
\[
A = 
\begin{pmatrix}
0 & -3 & 0 & 1 \\
3 & 0 & 0 & 6 \\
0 & 0 & 0 & 0 \\
-1 & -6 & 0 & 0 \\
\end{pmatrix}. 
\]

Then $\{J \subseteq [4] \; \vert \; \Pf(A_J) \in \R^\times\} = \{\emptyset, \{1, 2\}, \{1, 4\}, \{2, 4\}\}$. 
Using the same notation from Example~\ref{exm:representable}, we find that $\cB = \cB' \Delta \{3\} = \{\emptyset, \{1, 2\}, \{1, 4\}, \{2, 4\}\}$. 
\end{example}

\begin{remark}
In general, one can choose $T  = \emptyset$ in the representation if and only if $\emptyset$ is a basis. In this case, we say that the orthogonal matroid is {\emph{normal}}. 
\end{remark}

\section{Orthogonal Matroids and Orthogonal Grassmannians}\label{section:Wick}

Let $P$ be a partial field and let $N = 2^n - 1$. Our goal for this section is to connect the notions of strong orthogonal matroids, weak orthogonal matroids, and representable matroids over $P$.  

We begin with the following lemma on normal orthogonal matroids. 

\begin{lemm}\label{lemma:normal}
Suppose $M$ is a normal orthogonal matroid. If $\emptyset \ne J \subseteq [n]$ is a basis, there exists another basis $J' \subseteq J$ with $|J'| = |J| -2$. 
\end{lemm}

\begin{proof}
We apply the symmetric exchange axiom to the bases $J$ and $\emptyset$. Pick some $a \in J$. Then there exists $a \ne b \in J = J \Delta \emptyset$ such that $J' = J \Delta\{a,b\}$ is a basis, and $|J'| = |J| - 2$. 
\end{proof}

Using this lemma, we now prove the desired result for normal orthogonal matroids.

\begin{theorem}\label{theorem:normalWick}
The followings are equivalent for a point $p \in \bfP^{N}(P)$ with $p_\emptyset = 1$:

\begin{enumerate}
\item There exists a skew-symmetric matrix $A$ over $P$ such that $p_J = \Pf(A_J)$ for all $J \subseteq [n]$. 
\item $p$ is a strong $P$-orthogonal matroid. 
\item $p$ is a weak $P$-orthogonal matroid. 
\end{enumerate}
\end{theorem}

\begin{proof}
(1) $\Rightarrow$ (2) follows from Proposition~\ref{proposition:Pfaffian}, and (2) $\Rightarrow$ (3) follows from Proposition~\ref{proposition:supp}. So it suffices to prove that (3) $\Rightarrow$ (1). 

Given a point $p = (p_J) \in\bfP^{N}(P)$ with support equal to the orthogonal matroid $M$ and satisfying the $4$-term Wick equations, let $A = (a_{ij})$ be the skew-symmetric matrix defined by
\[
a_{ij} = p_{\{i,j\}} \in P, \; 1 \leqslant i < j \leqslant n.
\]
We claim that $p_J = \Pf(A_J)$ for all $J \subseteq [n]$. 

When $|J|$ is odd, or $|J| = 0$ or $2$, we clearly have $p_J = \Pf(A_J)$. Now suppose $J = \{i_1, i_2, i_3, i_4\}$, where $1 \leqslant i_1 < i_2 < i_3 < i_4 \leqslant n$. Let $J_1 = \{i_1\}$ and $J_2 = \{i_2, i_3, i_4\}$. Since $p$ satisfies the $4$-term Wick equations, we have 
\[
p_{\emptyset} p_{J} - p_{\{i_1, i_2\}} p_{\{i_3, i_4\}} + p_{\{i_1, i_3\}} p_{\{i_2,i_4\}} - p_{\{i_1,i_4\}} p_{\{i_2,i_3\}} = 0. 
\]
But $p_J = \Pf(A_J)$ when $|J| = 2$. Therefore, by the recursive definition of Pfaffians, $p_J = \Pf(J)$ when $|J| = 4$. 

Now suppose $p_J = \Pf(A_J)$ for bases $J$ with $|J| = 0, 2, 4, \dots, 2r-2$. Let $J = \{i_1, \dots, i_{2r}\}$ be a basis of $M$ with $1 \leqslant i_1 < \dots < i_{2r} \leqslant n$. By Lemma ~\ref{lemma:normal}, there exists another basis $J' \subseteq J$ with $|J'| = 2r -4$. Without loss of generality $J \backslash J'= \{i_1, i_2, i_3, i_4\}$. 

Let $J_1 = \{i_1, i_5, \dots, i_{2r}\}$ and let $J_2 = \{i_2, i_3, \dots, i_{2r}\}$. By the $4$-term Wick relations, we have 
\[
p_{J'}p_{J} - p_{J_1 \cup \{i_2\}}p_{J_2 \backslash \{i_2\}} + p_{J_1 \cup \{i_3\}}p_{J_2 \backslash \{i_3\}} - p_{J_1 \cup \{i_4\}}p_{J_2 \backslash \{i_4\}} = 0. 
\]
Since $p_{J'} \ne 0$, by Proposition~\ref{proposition:Pfaffian} and induction, we obtain that $p_J = \Pf(A_J)$. 

Finally, suppose $J = \{i_1, \dots, i_{2r}\} \subseteq [n]$ is not a basis for $M$. If there exists $\{a, b, c, d\} \subseteq J$ such that $J' = J \backslash \{a,b,c,d\}$ is a basis, then the same proof would apply, giving $\Pf(A_J) = p_J = 0$. 
Otherwise, we have $p_{J'} = 0$ for all $J' \subseteq J$ with $|J'| = |J| - 4$. Therefore, 
\[
\begin{split}
\Pf(A_J) & = \sum_{j = 2}^{2r} (-1)^j \cdot p_{\{i_1, i_j\}} \cdot p_{\{i_2, \dots, \hat{i_j}, \dots, i_{2r}\}} \\
& = \sum_{j = 2}^{2r} (-1)^j \cdot p_{\{i_1, i_j\}} \cdot \left(\sum_{j \ne k = 3}^{2r} (-1)^k \cdot p_{\{i_2, i_k\}} \cdot p_{\{i_3, \dots, i_{2r}\} \backslash \{i_j, i_k\}} \right)\\
& = 0. 
\end{split}
\]
\end{proof}

We now turn to the general case. 

\begin{theorem}\label{theorem:PFWick}
The following are equivalent for arbitrary point $p \in \bfP^{N}(P)$: 
\begin{enumerate}
\item There exists a skew-symmetric matrix $A$ over $P$ such that $p_J = \Pf(A_J)$ for all $J \subseteq [n]$. 
\item $p$ is a strong $P$-orthogonal matroid. 
\item $p$ is a weak $P$-orthogonal matroid 
\end{enumerate}
\end{theorem}

\begin{proof}
Again, it suffices to prove (3) $\Rightarrow$ (1). Let $p = (p_J)$ be a point in $\bfP^N(P)$ satisfying the $4$-term Wick equations and assume that $\Supp(p)$ is the set of bases of an orthogonal matroid. Pick $p_{T} \ne 0$ and without loss of generality $p_T = 1$. Consider the point $q = (q_J) = (p_{J \Delta T}) \in \bfP^N(P)$. Since $q_\emptyset = 1$ and $q$ satisfies all the $4$-term Wick equations, it follows from Theorem~\ref{theorem:normalWick} that $q$ satisfies {\em all} of the Wick equations. So $p$ satisfies all of the Wick equations as well. 
\end{proof}

\section{The Number of Representable Orthogonal Matroids}\label{section:how-many}

In this section, we establish an upper bound for the number of orthogonal matroids on $[n]$ which are representable over some partial field. For this, we use a theorem of Nelson~\cite{Ne18} concerning zero patterns of a collection of polynomials. 

\medskip

For a polynomial $f \in \Z[x_1, \dots, x_m]$, we write $||f||$ for the maximal absolute value of the coefficients of $f$. In particular, $||0|| = 0$. 
Let $k$ be a field, and let $\psi_k: \Z[x_1, \dots, x_m] \to k[x_1, \dots, x_m]$ be the ring homomorphism induced by the natural homomorphism $\varphi_k: \Z \to k$. 

\medskip

Let $f_1, \dots, f_N \in \Z[x_1, \dots, x_m]$. We say a set $S \subseteq [N]$ is {\emph{realizable}} with respect to $\{f_1, \dots, f_N\}$ if there is a field $k$ and a vector ${\bf{u}} \in k^m$ such that 
\[
S = \{i \in [N] \; \vert \; \psi_k(f_i)({\bf{u}}) \ne 0 \}. 
\]

\begin{theorem}[Nelson]\label{theorem:Nelson}
Let $c, d \in \Z$ and let $f_1, \dots, f_N \in \Z[x_1, \dots, x_m]$ with $\deg(f_i) \leq d$ and $||f_i|| \leq c$ for all $i$. Let $\log$ denote the logarithm to base $2$. If $r$ satisfies 
\[
r > \binom{Nd + m}{m} (\log(3r) + N\log(c(eN)^d)), 
\]
then $\{f_1, \dots, f_N\}$ has at most $r$ realizable sets. 
\end{theorem}

\begin{theorem}
If $n \geq 12$, then the number of normal orthogonal matroids on [n] representable over some field is at most $2^{n^3}$. 
\end{theorem}

\begin{proof}
Let $A$ be an $n \times n$ skew-symmetric matrix of indeterminates, where $a_{i, i} = 0$ and $a_{i, j} = x_{i,j}$ for $i > j$. For ease of notation, we relabel the indeterminates as $x_1, \dots, x_m$, where $m = \frac{n(n-1)}{2}$. 
Let $\{f_1, \dots, f_N\}$ be the Pfaffians of all square submatrices of $A$ of even size, where 
\[
N = \sum_{0 \leq k \leq n, \text{ $k$ is even}}\binom{n}{k} = 2^{n-1}. 
\]
Notice that $\deg f_i \leq n-1$ and $||f_i|| = 1$ for every $i = 1, \dots, N$. 

If $M$ is a $k$-representable normal orthogonal matroid on $[n]$, there exists a skew-symmetric matrix $A$ over $k$ such that $\cB = \{J \subseteq E \; \vert \; A_{J} \text{ is nonsingular}\}$. 
This means precisely that $\cB$ is realizable with respect to $\{f_1, \dots, f_N\}$, so it suffices to show the number of realizable sets for $\{f_1, \dots, f_N\}$ is at most $2^{n^3}$. 

In Theorem~\ref{theorem:Nelson}, let $c = 1, d = n-1, N = 2^{n-1}$, and $r = 2^{n^3}$, where $n \geq 12$. Then by standard inequalities, we have
\[
\begin{split}
\binom{Nd + m}{m} & = \binom{2^{n-1}(n-1) + \frac{n(n-1)}{2}}{\frac{n(n-1)}{2}}\\
& \leq \binom{n \cdot 2^n}{n^2} \\ 
& \leq \left(\frac{e\cdot n \cdot 2^{n-1}}{n^2}\right)^{n^2} \\
& < \left(\frac{2^{n+1}}{n}\right)^{n^2}, 
\end{split}
\]
and 
\[
\begin{split}
\log(3r) + N\log(c(eN)^d) & = \log 3 + n^3 + (n-1)\cdot 2^{n-1}(\log e + n-1) \\
& \leq 2 + n^3 + (n-1)\cdot 2^{n-1}(n+1)\\
& < n^2 \cdot 2^{n-1}. 
\end{split}
\]
Therefore, 
\[
\begin{split}
\binom{Nd + m}{m} (\log(3r) + N\log(c(eN)^d)) & < \left(\frac{2^{n+1}}{n}\right)^{n^2} \cdot n^2 \cdot 2^{n-1} \\
& = 2^{n^3} \cdot \frac{2^{n^2 + n - 1}}{n^{n^2 - 2}} < 2^{n^3} = r. 
\end{split}
\]
By Theorem~\ref{theorem:Nelson}, $\{f_1, \dots, f_N\}$ has at most $r = 2^{n^3}$ realizable sets, so there are at most $2^{n^3}$ normal orthogonal matroids on $[n]$ which are representable over some field.
\end{proof}

\begin{coro}
Asymptotically 100\% of all orthogonal matroids are not representable over any field.
\end{coro}

\begin{proof}
Let $k$ be a field and let be a $k$-representable orthogonal matroid. Then for any $T \in \cB$, $M \Delta T$ is a normal representable orthogonal matroid. This shows the number of representable orthogonal matroid on $[n]$ is at most $2^n \cdot 2^{n^3}$. The result now follows from Theorem~\ref{theorem:numberofmatroids}.
\end{proof}

\begin{lemm}
Let $\varphi: P_1 \to P_2$ be a homomorphism of partial fields. If an orthogonal matroid $M$ is $P_1$-representable, then $M$ is also $P_2$-representable. 
\end{lemm}

\begin{proof}
Let $A = (a_{ij})$ be a skew-symmetric matrix over $P_1$ and $T$ a subset of $E$ such that the pair $(A,T)$ represents $M$. 
Consider the new matrix $B = \varphi(A) := (\varphi(a_{ij}))$. If $\Pf(A_J) = 0$ then $\Pf(B_J) = \varphi(0) = 0 \in P_2$. If $\Pf(A_J) \neq 0$, then $\Pf(B_J) = \varphi(\Pf(A_J)) \neq 0$. This shows that the pair $(B,T)$ represents $M$. 
\end{proof}

Combined with Lemma~\ref{lemma:homomtofield}, this yields:

\begin{coro}
If an orthogonal matroid is representable over a partial field $P$, then it is representable over some field $k$. 
\end{coro}

We conclude:

\begin{coro}
Asymptotically 100\% of all orthogonal matroids are not representable over any partial field.
\end{coro}

\end{document}